\documentclass[reqno]{amsart}

\usepackage{amsmath,amssymb,amsthm,amsfonts,mathtools}
\usepackage{mathrsfs}
\usepackage{xcolor}
\usepackage{hyperref}
\usepackage[noabbrev]{cleveref}

\makeatletter
\@namedef{subjclassname@2020}{%
    \textup{2020} Mathematics Subject Classification}
\makeatother

\newcommand{\R}{\mathbb R}
\newcommand{\dd}{\,d}
\newcommand{\calC}{\mathcal C}

\newcommand{\calF}{\mathcal F}
\newcommand{\calT}{\mathcal T}

\DeclareSymbolFont{largesymbolsA}{U}{esint}{m}{n}
\DeclareMathSymbol{\fintop}{\mathop}{largesymbolsA}{'037}
\newcommand{\fint}{\fintop\nolimits}

\theoremstyle{plain}
\newtheorem{theorem}{Theorem}[section]
\newtheorem{lemma}[theorem]{Lemma}
\newtheorem{proposition}[theorem]{Proposition}
\newtheorem{corollary}[theorem]{Corollary}

\theoremstyle{definition}
\newtheorem{definition}[theorem]{Definition}

\theoremstyle{remark}
\newtheorem{remark}[theorem]{Remark}

\numberwithin{equation}{section}

\crefname{theorem}{Theorem}{Theorems}
\Crefname{theorem}{Theorem}{Theorems}
\crefname{lemma}{Lemma}{Lemmas}
\Crefname{lemma}{Lemma}{Lemmas}
\crefname{proposition}{Proposition}{Propositions}
\Crefname{proposition}{Proposition}{Propositions}
\crefname{corollary}{Corollary}{Corollaries}
\Crefname{corollary}{Corollary}{Corollaries}
\crefname{definition}{Definition}{Definitions}
\Crefname{definition}{Definition}{Definitions}
\crefname{assumption}{Assumption}{Assumptions}
\Crefname{assumption}{Assumption}{Assumptions}
\crefname{remark}{Remark}{Remarks}
\Crefname{remark}{Remark}{Remarks}
\crefname{section}{Section}{Sections}
\Crefname{section}{Section}{Sections}

\crefformat{section}{Section~#2#1#3}
\Crefformat{section}{Section~#2#1#3}
\crefformat{subsection}{Section~#2#1#3}
\Crefformat{subsection}{Section~#2#1#3}
\crefformat{theorem}{Theorem~#2#1#3}
\Crefformat{theorem}{Theorem~#2#1#3}
\crefformat{lemma}{Lemma~#2#1#3}
\Crefformat{lemma}{Lemma~#2#1#3}
\crefformat{proposition}{Proposition~#2#1#3}
\Crefformat{proposition}{Proposition~#2#1#3}
\crefformat{corollary}{Corollary~#2#1#3}
\Crefformat{corollary}{Corollary~#2#1#3}
\crefformat{definition}{Definition~#2#1#3}
\Crefformat{definition}{Definition~#2#1#3}
\crefformat{assumption}{Assumption~#2#1#3}
\Crefformat{assumption}{Assumption~#2#1#3}
\crefformat{remark}{Remark~#2#1#3}
\Crefformat{remark}{Remark~#2#1#3}

\title[Harnack estimates for nonlocal double phase functionals]
{Harnack-type estimates for nonlocal double phase functionals}

\author[Y. Fang and C. Zhang \hfil \hfilneg]{Yuzhou Fang and Chao Zhang$^*$}

\thanks{*Corresponding author}

\address{Yuzhou Fang \hfill\break School of Mathematics, Harbin Institute of Technology, Harbin 150001, China}
\email{18b912036@hit.edu.cn}

\address{Chao Zhang \hfill\break School of Mathematics and Institute for Advanced Study in Mathematics, Harbin Institute of Technology, Harbin 150001, China}
\email{czhangmath@hit.edu.cn}

\hypersetup{
    colorlinks=false,
    pdfborder={0 0 1},
    linkbordercolor={1 0 0},
    citebordercolor={0 1 0},
    urlbordercolor={0 1 1},
    pdftitle={Harnack-type estimates for nonlocal double phase functionals},
    pdfauthor={Yuzhou Fang and Chao Zhang},
    pdfsubject={Regularity theory for nonlocal double phase variational problems},
    pdfcreator={LaTeX with hyperref},
    pdfkeywords={Local upper estimates; Harnack-type estimates; nonlocal double phase functionals; intrinsic tails}
}

\begin{document}

\begin{abstract}
We establish Harnack-type inequalities for local minimizers of nonlocal double
phase functionals, involving an explicit two-radius positive-part tail and a
negative far-field tail. The key difficulty is the mismatch between the
natural \((p,q)\)-growth scales in the available modular supremum estimate and
the low integrability exponent provided by the weak Harnack inequality. To
overcome this obstacle, we establish a concentric modular estimate and derive
from it a local upper estimate valid for every integrability exponent. To the
best of our knowledge, our results provide the first Harnack-type theory for
such nonlocal double phase functionals under the natural structural
conditions.
\end{abstract}

\maketitle

\section{Introduction}

Let \(\Omega\subset\R^n\), \(n\ge2\), be open and bounded. In this paper, we
study local minimizers of the nonlocal double phase functional
\begin{equation}\label{main}
    \calF(u;\Omega)
    :=\iint_{\calC_\Omega}\left(|u(x)-u(y)|^p K_{sp}(x,y)+ a(x,y)|u(x)-u(y)|^q K_{tq}(x,y)
    \right)\dd x\!\dd y ,
\end{equation}
where $a(x,y)\ge0$,
\[
  \calC_\Omega=(\R^n\times\R^n)\setminus((\R^n\setminus\Omega)\times(\R^n\setminus\Omega)).
\]
The kernels \(K_{sp}\) and \(K_{tq}\) are symmetric and uniformly comparable
to \(|x-y|^{-n-sp}\) and \(|x-y|^{-n-tq}\), respectively. When \(a\equiv0\),
the functional reduces to a fractional \(p\)-type energy, whereas a nonzero
coefficient couples the two fractional phases. Since \(a\) may vanish, the
ellipticity can change dramatically with position, creating nontrivial challenges that
are absent from the fractional \(p\)-Laplacian analogue.

We impose the following structural assumptions throughout the paper. The
exponents satisfy
\begin{equation}\label{eq:exponents}
    0<s,t<1,\qquad 1<p\le q<\infty.
\end{equation}
The fractional kernels are assumed to meet the standard ellipticity bounds
\begin{equation}\label{eq:kernel-bounds}
    \frac{\Lambda^{-1}}{|x-y|^{n+sp}}
    \le K_{sp}(x,y)\le
    \frac{\Lambda}{|x-y|^{n+sp}},\quad
    \frac{\Lambda^{-1}}{|x-y|^{n+tq}}
    \le K_{tq}(x,y)\le
    \frac{\Lambda}{|x-y|^{n+tq}},
\end{equation}
for some \(\Lambda\ge1\), and the measurable coefficient $a:\mathbb{R}^n\times\mathbb{R}^n\rightarrow\mathbb{R}$ satisfies
\begin{equation}\label{eq:a-basic}
  a\in L^\infty_{\rm loc}(\Omega\times\Omega) \ \ \text{and}\ \  0\le a(x,y)=a(y,x)
    \quad\text{for a.e. }(x,y)\in\R^n\times\R^n.
\end{equation}
For the local upper and Harnack-type estimates, we further impose
\begin{equation}\label{eq:a-regular}
\begin{cases}
|a(x_1,y_1)-a(x_2,y_2)|\le [a]_\alpha(|x_1-x_2|+|y_1-y_2|)^\alpha, \quad &\text{if } s\le t,  \\[2mm]
|a(x,y)|\le M:=\|a\|_{L^\infty}, \quad &\text{if } s>t.
\end{cases}
\end{equation}
where, in the case \(s\le t\), \(\alpha\in(0,1]\) and the first inequality is
imposed on \(\R^n\times\R^n\). We also require the Sobolev admissibility condition
\begin{equation}\label{eq:sobolev-range}
    q\le\frac{np}{n-sp}
    \quad\text{if }sp<n,
    \qquad
    p\le q<\infty
    \quad\text{if }sp\ge n.
\end{equation}
This range is contained in the one for which local boundedness of minimizers
follows from~\cite[Theorem~2.9]{FKZ26}; see \Cref{lem2-1} below.

The regularity theory for nonlocal problems has developed rapidly over the
past two decades. For the fractional Laplace equation, Caffarelli and
Silvestre~\cite{CS07} established the Harnack inequality through the extension
method; for the parabolic setting, see~\cite{Str19,KW24}. In the nonlinear
setting, Di Castro, Kuusi and Palatucci~\cite{DKP14,DKP16} extended the
De Giorgi--Nash--Moser theory to the fractional \(p\)-Laplacian and obtained
H\"older continuity, weak Harnack inequalities, and Harnack inequalities. A
key ingredient in their analysis is a logarithmic estimate adapted to the
nonlocal setting.


More recently, these regularity theories have also been extended to nonlocal problems with
nonstandard growth. For nonlocal \(G\)-Laplace equations with Orlicz growth,
Byun, Kim and Ok~\cite{BKO} studied local boundedness and H\"older continuity
by adapting the techniques developed in~\cite{DKP16}; see also~\cite{CKW22}
for an alternative approach. Based on weak Harnack estimates and local boundedness of subsolutions, Fang and Zhang \cite{FZ23} further obtained a Harnack-type inequality for such equations, which was subsequently refined in \cite{CKW23,BKS23}, including the analysis of the stability of the full Harnack estimate as \(s\to1\). 

Among nonlocal problems with nonstandard growth, nonlocal double phase
functionals form a particularly challenging class due to the sharp interaction between different growth powers and fractional differentiability orders.
Double phase energies originate in the variational theory of functionals with
nonstandard growth, and their local regularity theory, including Harnack-type
estimates, has been explored extensively; see, for
instance,~\cite{BCM15,CM15a,CM15b}. For the nonlocal counterpart, De Filippis
and Palatucci~\cite{DP19} initiated the study of these problems in the
viscosity framework, while a self-improving property for bounded weak
solutions was later established in~\cite{SM22}. The relation between weak and
viscosity solutions was further investigated in~\cite{FZ23,GLZ26}. These
results concern the case \(s\ge t\), where the \(q\)-phase acts as a
lower-order perturbation. The complementary range \(s\le t\), where the
higher-growth phase is associated with the higher differentiability order, is
substantially more delicate. Byun, Ok and Song~\cite{BOS22} obtained local
H\"older continuity for minimizers and weak solutions of~\eqref{main} under
the assumption \(0\le a\in L^\infty(\R^n\times\R^n)\) and the balance
condition \(tq\le sp+\alpha\).

For nonlocal double phase equations with a strictly positive bounded
coefficient in the regime \(sp\ge tq\), a full nonlocal Harnack inequality was
obtained in~\cite{Kim25}. The strict positivity assumption, however, excludes
the possibly vanishing coefficients considered here. Related mixed-order
equations with two possibly degenerate coefficients were studied recently by
Lee, Ok and Song~\cite{LOS26}. Their operator combines two fractional phases
of different differentiability orders but with the same growth exponent
\(p\); they obtained regularity results and, under the assumption that one
coefficient is identically one, a Harnack inequality.

The present problem is structurally different: it permits genuine double
phase growth \(p<q\) and \(s\le t\), while also covering the equal-growth case
\(p=q\), and requires modular estimates adapted simultaneously to distinct
growth exponents and differentiability orders. Weak Harnack inequalities for
this class were established in~\cite{FKZ26}.

Despite these developments, no Harnack-type estimate appears to be available
in the genuinely nonstandard-growth case \(p<q\) for the nonlocal
functional~\eqref{main} when the modulating coefficient may vanish. The
obstruction is that the modular supremum estimate
in~\cite[Theorem~2.10]{FKZ26} is expressed at the natural \(p\)- and
\(q\)-growth scales, whereas the weak Harnack theorem provides only a
generally small integrability exponent. These estimates therefore cannot be
combined directly. What is needed is an upper estimate of
\(L^\ell\)-to-\(L^\infty\) form, valid for every \(\ell>0\), that retains the
exterior contribution at the radius required by the Harnack argument.

The main novelty of this paper is a scale-compatible procedure that closes this gap.  We refine the modular supremum argument of~\cite[Theorem~2.10]{FKZ26} to obtain a low-integrability local upper estimate for every \(\ell>0\).  The proof keeps the exterior integration radius fixed, produces an absorbable supremum on a larger ball through modular interpolation, and closes the resulting recursion by a hole-filling argument.  The resulting flexibility permits the upper estimate to be applied with the small exponent supplied by the weak Harnack
inequality.  The iteration also identifies the two-radius intrinsic tail \(\calT_{\rho,R}\): \(\rho\) determines where the exterior interaction begins, while \(R\) fixes the coefficient and normalization scales in the intrinsic modular.  This separation preserves the scales required by the upper and weak Harnack estimates.  The two radii decouple the exterior cutoff scale from the intrinsic reference scale; they do not separate the \(p\)- and \(q\)-phases.

Before stating the main results, we introduce the intrinsic scale functions.
For \(\tau\ge0\), define
\begin{equation*}
h_r(\tau):=\frac{\tau^{p-1}}{r^{sp}}
    +a^-_r\frac{\tau^{q-1}}{r^{tq}}
\end{equation*}
and
\begin{equation*}
H_r(\tau):=\frac{\tau^p}{r^{sp}}
    +a^-_r\frac{\tau^q}{r^{tq}}
\end{equation*}
with
\begin{equation*}
a^-_r:=\inf_{B_r(x_0)\times B_r(x_0)} a(x,y).
\end{equation*}
We denote their inverses by \((H_r)^{-1}\) and \((h_r)^{-1}\),
respectively. Because the coefficient \(a\) depends on both variables, we use
the weighted tail
\[
\mathrm{Tail}_a(v;x_0,r,R)
:=\sup_{x\in B_R(x_0)}
\int_{\R^n\setminus B_r(x_0)}
a(x,y)\frac{|v(y)|^{q-1}}{|y-x_0|^{n+tq}}\dd y.
\]
The unweighted \(p\)-tail is defined by
\[
\mathrm{Tail}(v;x_0,r)
:=\int_{\R^n\setminus B_r(x_0)}
\frac{|v(y)|^{p-1}}{|y-x_0|^{n+sp}}\dd y.
\]
Set
\begin{align}\label{eq:two-radius-tail}
    \calT_{r,R}(v;x_0)
    :=\mathrm{Tail}(v;x_0,r)+\mathrm{Tail}_a(v;x_0,r,R).
\end{align}
We abbreviate \(\calT_R(v;x_0):=\calT_{R,R}(v;x_0)\).

We now state the local upper estimate that provides the missing ingredient for
combining the modular supremum estimate with the weak Harnack inequality.

\begin{theorem}[Local upper estimate]\label{thm1}
Assume that \eqref{eq:exponents}--\eqref{eq:sobolev-range} hold and
\begin{align*}
\begin{cases}
tq \leq sp+\alpha,  \quad &\text{for  } s\le t,\\[2mm]
tq\le sp,        \quad&\text{for  } s>t.
\end{cases}
\end{align*}
Let \(0<R\le1\),  \(B_R(x_0)\Subset\Omega\), and let
\(u\in\mathcal{A}(\Omega)\cap L^{p-1}_{sp}(\R^n)
\cap L^{q-1}_{a,tq}(\Omega,\R^n)\) be a local minimizer
of~\eqref{main} in \(\Omega\). Then, for every \(\ell>0\), there exists a
constant \(C_\ell\ge1\), depending only on \(\ell\) and
\(\mathfrak D_u(B_R(x_0))\), such that
\begin{equation}\label{eq:local-bound}
    \sup_{B_{R/2}(x_0)}u\le C_\ell\left(\fint_{B_R(x_0)}u_+^\ell\dd x\right)^{1/\ell}
    +C_\ell\,(h_R)^{-1}\bigl(\calT_{R/2,R}(u_+;x_0)\bigr).
\end{equation}
\end{theorem}

For clarity, we restate the weak Harnack inequality
from~\cite[Theorem~2.5]{FKZ26} in the normalization used here.

\begin{proposition}[Weak Harnack estimate]
\label{pro1}
Assume that \eqref{eq:exponents}--\eqref{eq:sobolev-range} hold and
\begin{align*}
\begin{cases}
\alpha<tq \leq sp+\alpha,  \quad &\text{for  } s\le t,\\[2mm]
tq\le sp,        \quad&\text{for  } s>t.
\end{cases}
\end{align*}
Let \(u\in\mathcal{A}(\Omega)\cap L^{p-1}_{sp}(\R^n)
\cap L^{q-1}_{a,tq}(\Omega,\R^n)\) be a local minimizer
of~\eqref{main} in \(\Omega\). Let \(0<R\le1\), suppose that
\(B_{2R}(x_0)\Subset\Omega\) and \(u\ge0\) a.e.\ in
\(B_{2R}(x_0)\).
Then there exist constants \(\varepsilon\in(0,1)\) and \(C\ge1\), depending
only on \(\mathfrak D_u(B_{2R}(x_0))\), such that
\begin{equation}\label{eq:weak-harnack}
    \left(\fint_{B_{R/2}(x_0)}u^\varepsilon\dd x\right)^{1/\varepsilon}
    \le
    C\left[
        \inf_{B_{R/2}(x_0)}u+(h_{2R})^{-1}\bigl(\calT_{2R}(u_-;x_0)\bigr)
    \right].
\end{equation}
\end{proposition}

\begin{remark}
The key feature of \Cref{thm1} is its validity for every \(\ell>0\), including \(\ell<p-1\).  This removes the exponent mismatch and
allows the weak Harnack exponent to be used directly in the upper estimate. 
\end{remark}

Combining the local upper estimate with the weak Harnack inequality yields the
following Harnack-type estimate.

\begin{theorem}[Harnack-type inequality]
\label{thm:harnack-tail}
Assume that \eqref{eq:exponents}--\eqref{eq:sobolev-range} hold and
\begin{align*}
\begin{cases}
\alpha<tq \leq sp+\alpha,  \quad &\text{for  } s\le t,\\[2mm]
tq\le sp,        \quad&\text{for  } s>t.
\end{cases}
\end{align*}
Let \(u\in\mathcal{A}(\Omega)\cap L^{p-1}_{sp}(\R^n)
\cap L^{q-1}_{a,tq}(\Omega,\R^n)\) be a local minimizer
of~\eqref{main} in \(\Omega\). Let \(0<R\le1\), suppose that
\(B_{2R}(x_0)\Subset\Omega\) and \(u\ge0\) a.e.\ in
\(B_{2R}(x_0)\).
Then there exists a constant \(C\ge1\), depending only on
\(\mathfrak D_u(B_{2R}(x_0))\), such that
\begin{equation}\label{eq:harnack-tail}
    \sup_{B_{R/4}(x_0)}u
    \le
    C\left[\inf_{B_{R/4}(x_0)}u+(h_{R/2})^{-1}\bigl(\calT_{R/4,R/2}(u_+;x_0)\bigr)+(h_{2R})^{-1}\bigl(\calT_{2R}(u_-;x_0)\bigr)\right].
\end{equation}
\end{theorem}

The preceding estimate has the following single-tail reformulation.

\begin{corollary}\label{cor:harnack-single-tail}
Under the assumptions of \Cref{thm:harnack-tail}, there exists a constant
\(C\ge1\), depending only on \(\mathfrak D_u(B_{2R}(x_0))\), such that
\begin{equation}\label{eq:harnack-single-tail}
    \sup_{B_{R/4}(x_0)}u\le
    C\left[\inf_{B_{R/4}(x_0)}u+(h_{2R})^{-1}\bigl(\calT_{R/4,2R}(|u|;x_0)\bigr)\right].
\end{equation}
\end{corollary}

\begin{remark}

\begin{itemize}
\item[(i)] The two radii in \(\calT_{\rho,R}\) have different roles:
\(\rho\) is the exterior cutoff radius, whereas \(R\) is the intrinsic
reference radius used for the coefficient infimum, the supremum in the
weighted tail, and the normalization through \(H_R\) and \(h_R\). Thus
the two-radius notation decouples the exterior interaction scale from the
intrinsic reference scale.
\smallskip

\item[(ii)] \Cref{cor:harnack-single-tail} combines the two terms into one
intrinsic tail at the common reference scale \(2R\). This form is more
compact, but it is slightly weaker than \eqref{eq:harnack-tail}: it loses the
sign separation and replaces the sharper positive reference scale \(R/2\) by
\(2R\). Thus it is a reformulation of the two-tail estimate, not an absorption
of the positive-part tail.

\end{itemize}
\end{remark}

Finally, by~\cite[Section~3]{BOS22}, the preceding results also hold for weak
solutions of the nonlocal double phase equation
\begin{align*}
&\mathrm{P.V.}\int_{\mathbb{R}^n}|u(x)-u(y)|^{p-2}(u(x)-u(y))K_{sp}(x,y)\dd y \\
&\quad+\mathrm{P.V.}\int_{\mathbb{R}^n}a(x,y)
|u(x)-u(y)|^{q-2}(u(x)-u(y))K_{tq}(x,y)\dd y=0
\quad \text{in } \Omega,
\end{align*}
where $\mathrm{P.V.}$ stands for the Cauchy principal value.

The paper is organized as follows. In \Cref{sec:algebra-euler} we collect the
notation, definitions, and auxiliary lemmas. The central argument appears in
\Cref{sec:upper-harnack}, where we prove the flexible-radius modular estimate,
derive the arbitrary-exponent upper bound, and combine it with the weak
Harnack inequality.

\section{Preliminaries}
\label{sec:algebra-euler}

Unless explicitly stated otherwise, \(C\ge1\) denotes a constant that may
change from line to line. Dependencies on parameters are indicated in
parentheses; for example, \(C(p,q,n)\) means that \(C\) depends only on
\(p,q,n\). No such constant depends on the radii, levels, or iteration indices
under consideration.

For a subset \(E\subset\Omega\), we collect the relevant data as follows. If
\(s\le t\), let
\begin{equation*}
\mathfrak D_u(E):=\begin{cases}
\bigl(n,p,q,s,t,\Lambda,\alpha,[a]_\alpha,\|u\|_{L^\infty(E)}\bigr),
   &\text{if } sp\le n,\\[2mm]
\bigl(n,p,q,s,t,\Lambda,\alpha,[a]_\alpha,[u]_{W^{s,p}(E)}\bigr),
   &\text{if } sp> n.
\end{cases}
\end{equation*}
If \(s>t\), let
\begin{equation*}
\mathfrak D_u(E):=\begin{cases}
\bigl(n,p,q,s,t,\Lambda,\|a\|_{L^\infty},\|u\|_{L^\infty(E)}\bigr),
   &\text{if } sp\le n,\\[2mm]
\bigl(n,p,q,s,t,\Lambda,\|a\|_{L^\infty},[u]_{W^{s,p}(E)}\bigr),
   &\text{if } sp> n.
\end{cases}
\end{equation*}

For \(s\in(0,1)\) and \(p\ge1\), the fractional Sobolev space
\(W^{s,p}(\Omega)\) is defined by
\[
W^{s,p}(\Omega):=\left\{u\in L^p(\Omega)\Bigg|[u]_{W^{s,p}(\Omega)}:=\left(\int_\Omega\int_\Omega\frac{|u(x)-u(y)|^p}{|x-y|^{n+sp}}\,dx\,dy\right)^\frac{1}{p}<\infty\right\},
\]
and is equipped with the norm
\[
\|u\|_{W^{s,p}(\Omega)}=\|u\|_{L^p(\Omega)}+[u]_{W^{s,p}(\Omega)}.
\]
Let
\[
\mathcal{E}(u;\Omega):=\iint_{\mathcal{C}_\Omega}\left(\frac{|u(x)-u(y)|^p}{|x-y|^{n+sp}}+a(x,y)\frac{|u(x)-u(y)|^q}{|x-y|^{n+tq}}\right)\,dx\,dy,
\]
and define a function space related to minimizers of \eqref{main} as
\[
\mathcal{A}(\Omega):=\left\{u:\mathbb{R}^n\rightarrow\mathbb{R} \,\Big|\, u|_\Omega\in L^p(\Omega) \ \text{and } \mathcal{E}(u;\Omega)<\infty\right\}.
\]
By definition, \(\mathcal{A}(\Omega)\subset W^{s,p}(\Omega)\). Moreover, the
fractional Sobolev embedding theorem gives
\(\mathcal{A}(\Omega)\subset L^q(\Omega)\), provided
\(p\le q\le np/(n-sp)\) if \(sp<n\), or \(p\le q<\infty\) if \(sp\ge n\).

We next recall the definition of a minimizer of~\eqref{main}.

\begin{definition}
\label{def:minimizer}
A function \(u\in\mathcal{A}(\Omega)\) is a minimizer of~\eqref{main} if
\[
 \mathcal{F}(u;\Omega)\le\mathcal{F}(v;\Omega)
\]
for every \(v\in\mathcal{A}(\Omega)\) with \(v=u\) a.e.\ in
\(\mathbb{R}^n\setminus\Omega\).
\end{definition}

The existence of minimizers for~\eqref{main} was established
in~\cite[Section~3]{BOS22}. To account for the nonlocal character of the
functional, we introduce the tail space
\[
L^{p-1}_{sp}(\mathbb{R}^n):=\left\{u:\mathbb{R}^n\rightarrow\mathbb{R}\,\Bigg|\,\int_{\mathbb{R}^n}\frac{|u(x)|^{p-1}}{(1+|x|)^{n+sp}}\,dx<\infty\right\}.
\]
The \(q\)-growth term in~\eqref{main}, modulated by the coefficient \(a\),
also requires the weighted tail space
\[
L^{q-1}_{a,tq}(\Omega,\mathbb{R}^n):=\left\{u: \mathbb{R}^n\rightarrow\mathbb{R}\,\Bigg|\,\sup_{x\in \Omega}\int_{\mathbb{R}^n}a(x,y)\frac{|u(y)|^{q-1}}{(1+|y|)^{n+tq}}\,dy<\infty\right\}.
\]
The definitions of these tail spaces ensure that the tails introduced above
are finite. Throughout the paper, we write
\(u_+:=\max\{u,0\}\) and \(u_-:=\max\{-u,0\}\). For brevity, set
\[
    M_R^+(u;x_0):=\|u_+\|_{L^\infty(B_R(x_0))},
\]

We next recall the auxiliary results used below. Under
\eqref{eq:sobolev-range}, minimizers are locally bounded.

\begin{lemma}[\cite{FKZ26}]
\label{lem2-1}
Suppose that \eqref{eq:exponents}--\eqref{eq:a-basic} and
\eqref{eq:sobolev-range} hold. Then every minimizer
\(u\in\mathcal{A}(\Omega)\cap L_{sp}^{p-1}(\R^n)
\cap L_{a,tq}^{q-1}(\Omega,\R^n)\) of~\eqref{main} is locally bounded in
\(\Omega\).
\end{lemma}

We also recall the improved Caccioppoli inequality
from~\cite[Lemma~3.2]{FKZ26}, which plays the crucial role in \Cref{thm1}.

\begin{lemma}
\label{cac}
Assume that \eqref{eq:exponents}, \eqref{eq:kernel-bounds}, and
\eqref{eq:a-regular} hold. Let \(B_R=B_R(x_0)\subset\Omega\) with \(R\le1\),
and suppose that
\(u\in\mathcal{A}(\Omega)\cap L^{p-1}_{sp}(\R^n)
\cap L^{q-1}_{a,tq}(\Omega,\R^n)\) is a minimizer of~\eqref{main}. If
\(sp\le n\), assume in addition that \(u\) is bounded in \(B_R\). Then, for
\(w_\pm(x):=(u-k)_\pm(x)\) with
\(|k|\le\|u\|_{L^\infty(B_R)}\), we have
\begin{align*}
&\int_{B_\rho}\int_{B_\rho} \left(\frac{|w_\pm(x)-w_\pm(y)|^p}{|x-y|^{n+sp}}+a^-_R\frac{|w_\pm(x)-w_\pm(y)|^q}{|x-y|^{n+tq}}\right)\, dx\,dy\\
&\qquad+\int_{B_\rho}w_\pm(x)\left(\int_{\mathbb{R}^n}\frac{w_\mp^{p-1}(y)}{|x-y|^{n+sp}}+a(x,y)\frac{w_\mp^{q-1}(y)}{|x-y|^{n+tq}}\,dy\right)\,dx\\
&\le C\left(\frac{r}{r-\rho}\right)^{n+q}\Bigg[\int_{B_r} \left(\frac{w_\pm^p}{r^{sp}}+a^-_R\frac{w_\pm^q}{r^{tq}}\right)\,dx\\
&\qquad+\int_{B_r}w_\pm(y)\left(\int_{\mathbb{R}^n\setminus B_\rho}\frac{w^{p-1}_\pm(x)}{|x-x_0|^{n+sp}}+a(x,y)\frac{w^{q-1}_\pm(x)}{|x-x_0|^{n+tq}}\,dx\right)\,dy\Bigg]
\end{align*}
for \(R/2\le\rho<r\le R\), provided that
\begin{equation}
\label{exp}
\begin{cases}
tq\le sp+\alpha,  \quad &\text{for } sp\le n,\\[2mm]
tq\le n+\alpha+\left(s-\frac{n}{p}\right)q,  \quad &\text{for } sp>n.
\end{cases}
\end{equation}
Here \(C\ge1\) depends only on \(\mathfrak D_u(B_R)\).
\end{lemma}

The following is a standard fast geometric convergence lemma used in De Giorgi iteration; see, for example,~\cite[Lemma~7.1]{Giu03}.

\begin{lemma}\label{lem:fast-convergence}
Let \(A>0\), \(b>1\), and \(\gamma>0\). Suppose that a nonnegative sequence
\((Z_j)_{j\ge0}\) satisfies
\begin{equation*}
    Z_{j+1}\le A b^j Z_j^{1+\gamma},\qquad j\ge0.
\end{equation*}
If
\begin{equation*}
    Z_0\le A^{-1/\gamma}b^{-1/\gamma^2},
\end{equation*}
then 
\begin{equation*}
   Z_j\le A^{-1/\gamma}b^{-1/\gamma^2}b^{-j/\gamma}
   \qquad\text{for every }j\ge0,
\end{equation*}
and hence \(Z_j\to0\) as \(j\to\infty\).
\end{lemma}

\section{The upper estimate and proof of the Harnack-type inequalities}
\label{sec:upper-harnack}

In this section, we prove the Harnack-type inequalities. We first derive an
upper estimate at low integrability by a modular De Giorgi iteration with two
free concentric radii; this flexibility is needed to lower the initial
integrability exponent. To obtain the \(L^\ell\)-to-\(L^\infty\) estimate in
\Cref{thm1}, we then establish a modular interpolation inequality and apply a
standard hole-filling lemma. Finally, we combine \Cref{thm1} with the weak
Harnack inequality in \Cref{pro1}.

We begin with a concentric supremum estimate at the natural
\((p,q)\)-growth scales, following the strategy
of~\cite[Theorem~2.10]{FKZ26}.

\begin{lemma}[Concentric modular estimate]
\label{lem3-1}
Assume that \eqref{eq:exponents}--\eqref{eq:sobolev-range} hold. Let
\(B_R(x_0)\Subset\Omega\) with \(0<R\le1\), and let
\(u\in\mathcal{A}(\Omega)\cap L_{sp}^{p-1}(\R^n)
\cap L_{a,tq}^{q-1}(\Omega,\R^n)\) be a local minimizer of~\eqref{main} in
\(\Omega\). Then there exist constants \(C\ge1\) and \(\kappa>0\), depending
only on \(\mathfrak D_u(B_R(x_0))\), such that
\begin{equation}\label{eq:concentric-modular}
\begin{aligned}
    \sup_{B_\rho(x_0)}u_+
    \le C\left(\frac{R}{r-\rho}\right)^\kappa
    \Bigg[
        (H_R)^{-1}
        \left(\fint_{B_r(x_0)}H_R(u_+)\dd x\right)
        +(h_R)^{-1}
        \bigl(\calT_{R/2,R}(u_+;x_0)\bigr)
    \Bigg]
\end{aligned}
\end{equation}
for $\frac R2\le \rho<r\le R$, provided
\begin{equation*}
\begin{cases}
&tq\le sp+\alpha, \quad\text{if } s\le t,\\[2mm]
&tq\le sp, \quad\text{if } s>t.
\end{cases}
\end{equation*}
\end{lemma}

\begin{proof}
Throughout the proof, all balls are centered at \(x_0\). By
\Cref{lem2-1}, the minimizer \(u\) is locally bounded under the present
assumptions.


Let \(k>0\) be a number to be specified below, and set
\[
    Q:=\frac{R}{r-\rho}\ge2,\qquad
    r_j:=\rho+2^{-j}(r-\rho),\qquad
    k_j:=2(1-2^{-j-1})k,
\]
and
\[
    d_j:=k_{j+1}-k_j=2^{-j-1}k,\qquad
    w_j:=(u-k_j)_+,
\]
\[
    E_j:=B_{r_j}(x_0)\cap\{u>k_j\},\qquad
    Y_j:=\fint_{B_{r_j}(x_0)}H_R(w_j)\dd x,
    \qquad
    T:=\calT_{R/2,R}(u_+;x_0).
\]
Here $j\in \mathbb{N}\cup\{0\}$ and \(Q\ge2\) by \(r-\rho\le R/2\).  Moreover,
\[
r_j\downarrow\rho, \ \ k_j\uparrow2k\, \ \text{and}\ \  \ w_{j+1}\le w_j\le u_+.
\]
 If \(k_{j+1}\ge M_R^+(u;x_0)\) for some \(j\), then the conclusion is immediate because \(M_R^+(u;x_0)\le k_{j+1}<2k\).  We therefore consider
the remaining case, in which \(k_{j+1}<M_R^+(u;x_0)\) for every \(j\).

Define the interior energy by
\begin{align}
\mathscr E_{j+1}:={}&
\frac{1}{|B_{j+1}|}
\iint_{B_{j+1}\times B_{j+1}}
\frac{|w_{j+1}(x)-w_{j+1}(y)|^p}{|x-y|^{n+sp}}\dd y\!\dd x\notag\\
&+a_R^-\frac{1}{|B_{j+1}|}
\iint_{B_{j+1}\times B_{j+1}}
\frac{|w_{j+1}(x)-w_{j+1}(y)|^q}{|x-y|^{n+tq}}\dd y\!\dd x,
\label{E}
\end{align}
where \(B_j:=B_{r_j}(x_0)\). Applying \Cref{cac}, we obtain
\begin{align*}
   \mathscr E_{j+1}\le
    C\left(\frac{r_j}{r_j-r_{j+1}}\right)^{n+q}
    \bigl(I_j+J_j\bigr)\le C2^{j(n+q)}Q^{n+q}\bigl(I_j+J_j\bigr),
\end{align*}
where
\[
    I_j:=\fint_{B_j}
    \left(
        \frac{w_{j+1}^p}{r_j^{sp}}
        +a_R^-\frac{w_{j+1}^q}{r_j^{tq}}
    \right)\dd x
\]
and
\[
\begin{aligned}
    J_j:=\fint_{B_j}w_{j+1}(y)
    \int_{\R^n\setminus B_{j+1}}\bigg[
        \frac{w_{j+1}(x)^{p-1}}{|x-x_0|^{n+sp}}+a(x,y)\frac{w_{j+1}(x)^{q-1}}
                          {|x-x_0|^{n+tq}}
    \bigg]\dd x\!\dd y.
\end{aligned}
\]
Here we used
\[
 \frac{r_j}{r_j-r_{j+1}}\le 2^{j+1}\frac{R}{r-\rho}, \quad \frac{|B_j|}{|B_{j+1}|}
    =\left(\frac{r_j}{r_{j+1}}\right)^n\le2^n,
\]
since \(R/2\le r_{j+1}<r_j\le R\).

We next estimate \(I_j\) and \(J_j\). For every \(\tau\ge0\),
\begin{align*}
    \frac{\tau^p}{r_j^{sp}}
       +a_R^-\frac{\tau^q}{r_j^{tq}}
    &={}
    \left(\frac{R}{r_j}\right)^{sp}\frac{\tau^p}{R^{sp}}
    +\left(\frac{R}{r_j}\right)^{tq}
       a_R^-\frac{\tau^q}{R^{tq}}\\
    &\le2^{\max\{sp,tq\}}H_R(\tau).
\end{align*}
It follows that
\[
    I_j\le C\fint_{B_j}H_R(w_{j+1})\dd x.
\]
To estimate \(J_j\), first observe that \(k_{j+1}\ge0\), and hence \(w_{j+1}\le u_+\) on \(\R^n\).  For every \(y\in B_j\subset B_R\), the
symmetry of \(a\) and the inclusion \(\R^n\setminus B_{j+1}\subset\R^n\setminus B_{R/2}\) give
\begin{align*}
\int_{\R^n\setminus B_{j+1}}
    a(x,y)\frac{w_{j+1}(x)^{q-1}}{|x-x_0|^{n+tq}}\dd x &\le
\int_{\R^n\setminus B_{R/2}}
    a(y,x)\frac{u_+(x)^{q-1}}{|x-x_0|^{n+tq}}\dd x\\
&\le
\sup_{z\in B_R}
\int_{\R^n\setminus B_{R/2}}
    a(z,x)\frac{u_+(x)^{q-1}}{|x-x_0|^{n+tq}}\dd x.
\end{align*}
The corresponding \(p\)-phase integral is bounded directly by the first term in \(T\).  Therefore the inner exterior integral defining \(J_j\) is at most \(T\), uniformly for \(y\in B_j\), and
\[
    J_j\le C T\fint_{B_j}w_{j+1}\dd y.
\]
Hence the preceding estimates yield
\begin{equation}\label{ej}
    \mathscr E_{j+1}
    \le{}
   C2^{j(n+q)}Q^{n+q}
    \left[\fint_{B_{r_j}(x_0)}H_R(w_{j+1})\dd x +\calT_{R/2,R}(u_+;x_0)\fint_{B_{r_j}(x_0)}w_{j+1}\dd x
    \right].
\end{equation}

At every point where \(w_{j+1}>0\), we have \(w_j>d_j\), and
\[
    h_R(d_j)w_{j+1}
    \le h_R(d_j)w_j
    \le h_R(w_j)w_j=H_R(w_j).
\]
Consequently,
\[
    \fint_{B_{r_j}}w_{j+1}\dd x
    \le \frac{Y_j}{h_R(d_j)}.
\]
The first average in \eqref{ej} is at most \(Y_j\), since
\(w_{j+1}\le w_j\) and \(H_R\) is increasing.  As a result,
\begin{equation}\label{eq:iteration-energy}
    \mathscr E_{j+1}
    \le C(2^jQ)^{n+q}
    \left(1+\frac{T}{h_R(d_j)}\right)Y_j.
\end{equation}

We next apply the fractional Sobolev inequality to the left-hand side
of~\eqref{eq:iteration-energy}.
Define
\[
    \beta_p^*=
    \begin{cases}
        \dfrac{n}{n-sp},&sp<n,\\
        \infty,&sp\ge n,
    \end{cases}
    \qquad
    \beta_q^*=
    \begin{cases}
        \dfrac{n}{n-tq},&tq<n,\\
        \infty,&tq\ge n.
    \end{cases}
\]
Choose once and for all
\[
    1<\beta<\min\{\beta_p^*,\beta_q^*\},
    \qquad
    \beta':=\frac{\beta}{\beta-1}.
\]

For \((\sigma,m)=(s,p)\) or \((t,q)\), introduce
\[
    \mathcal G_{\sigma,m}(v;B_\varrho)
    :=\fint_{B_\varrho}\int_{B_\varrho}
       \frac{|v(x)-v(y)|^m}{|x-y|^{n+\sigma m}}\dd y\!\dd x,
\]
and recall the well-known Sobolev inequality
\[
\begin{aligned}
    \left(\fint_{B_\varrho}|v|^{m\beta}\dd x\right)^{1/\beta}
    \le C\left[
        \varrho^{\sigma m}\mathcal G_{\sigma,m}(v;B_\varrho)
        +\fint_{B_\varrho}|v|^m\dd x
    \right].
\end{aligned}
\]
Then set \(B:=B_{r_{j+1}}\), \(\varrho:=r_{j+1}\), and \(w:=w_{j+1}\).  Since
\(\varrho\le R\), division by \(R^{\sigma m}\) yields
\[
\begin{aligned}
    \frac1{R^{\sigma m}}
    \left(\fint_B|w|^{m\beta}\dd x\right)^{1/\beta}
    \le C\left[
        \mathcal G_{\sigma,m}(w;B)
        +\frac1{R^{\sigma m}}\fint_B|w|^m\dd x
    \right].
\end{aligned}
\]
The truncation map \(\tau\mapsto(\tau-k_{j+1})_+\) is one-Lipschitz, so
\(w\in W^{s,p}(B)\).  If \(a_R^->0\), then the lower kernel bound and
\(a(x,y)\ge a_R^-\) on \(B_R\times B_R\) give
\[
    a_R^-[u]_{W^{t,q}(B_R)}^q
    \le C\int_{B_R}\int_{B_R}
        a(x,y)|u(x)-u(y)|^qK_{tq}(x,y)\dd y\dd x<\infty.
\]
The local \(L^q\)-integrability follows from \cref{eq:sobolev-range} and the \((s,p)\)-Sobolev embedding.  Thus \(u\in W^{t,q}(B_R)\), and the one-Lipschitz truncation gives \(w\in W^{t,q}(B)\).  We may therefore apply the preceding scaled estimate to the \((s,p)\)-phase and, when \(a_R^->0\), to the \((t,q)\)-phase.

Assume first that \(a_R^->0\).  Minkowski's inequality gives
\begin{align*}
    \left(\fint_BH_R(w)^\beta\dd x\right)^{1/\beta}
    &\le
    \frac1{R^{sp}}
    \left(\fint_Bw^{p\beta}\dd x\right)^{1/\beta}
    +\frac{a_R^-}{R^{tq}}
    \left(\fint_Bw^{q\beta}\dd x\right)^{1/\beta}\\
    &\le C\left[
        \mathcal G_{s,p}(w;B)
        +a_R^-\mathcal G_{t,q}(w;B)
        +\fint_BH_R(w)\dd x
    \right]\\
    &=C\left[
        \mathscr E_{j+1}
        +\fint_{B_{r_{j+1}}}H_R(w_{j+1})\dd x
    \right].
\end{align*}
If \(a_R^-=0\), applying only the \((s,p)\)-estimate gives the same final
inequality with every \((t,q)\)-term omitted.  Finally,
\(B_{r_{j+1}}\subset B_{r_j}\),
\(w_{j+1}\le w_j\), and \(H_R\) is increasing; hence
\[
\begin{aligned}
    \fint_{B_{r_{j+1}}}H_R(w_{j+1})\dd x
    \le
    \frac{|B_{r_j}|}{|B_{r_{j+1}}|}
    \fint_{B_{r_j}}H_R(w_j)\dd x=\left(\frac{r_j}{r_{j+1}}\right)^nY_j
    \le2^nY_j.
\end{aligned}
\]
Consequently,
\begin{equation}\label{eq:modular-sobolev-explicit}
    \left(\fint_{B_{r_{j+1}}}
        H_R(w_{j+1})^\beta\dd x\right)^{1/\beta}
    \le C\bigl(\mathscr E_{j+1}+Y_j\bigr).
\end{equation}
By definition, \(w_{j+1}=0\) a.e.\ on
\(B_{r_{j+1}}\setminus E_{j+1}\). H\"older's inequality gives
\[
\begin{aligned}
    Y_{j+1}
    &=\fint_{B_{r_{j+1}}}
        H_R(w_{j+1})\chi_{E_{j+1}}\dd x\\
    &\le
    \left(\fint_{B_{r_{j+1}}}
        H_R(w_{j+1})^\beta\dd x\right)^{1/\beta}
    \left(\frac{|E_{j+1}|}{|B_{r_{j+1}}|}\right)^{1/\beta'}.
\end{aligned}
\]
Combining this with \eqref{eq:modular-sobolev-explicit} proves
\begin{equation}\label{eq:iteration-sobolev}
    Y_{j+1}
    \le C(\mathscr E_{j+1}+Y_j)
    \left(\frac{|E_{j+1}|}{|B_{r_{j+1}}|}\right)^{1/\beta'}.
\end{equation}
Here the level gap also gives
\begin{equation}\label{eq:iteration-support}
    \frac{|E_{j+1}|}{|B_{r_{j+1}}|}
    \le C\frac{Y_j}{H_R(d_j)},
\end{equation}
because \(w_j=u-k_j>d_j\) on \(E_{j+1}\), and therefore
\[
    H_R(d_j)\frac{|E_{j+1}|}{|B_{r_j}|}
    \le
    \fint_{B_{r_j}}H_R(w_j)\dd x=Y_j.
\]

For \(0<\lambda\le1\),
\[
    H_R(\lambda\tau)\ge\lambda^qH_R(\tau),\qquad
    h_R(\lambda\tau)\ge\lambda^{q-1}h_R(\tau).
\]
Consequently,
\[
    H_R(d_j)\ge2^{-q(j+1)}H_R(k),\qquad
    h_R(d_j)\ge2^{-(q-1)(j+1)}h_R(k).
\]
Using \((2^jQ)^{n+q}\ge1\) and substituting
\eqref{eq:iteration-energy} and \eqref{eq:iteration-support}
into~\eqref{eq:iteration-sobolev}, we obtain
\[
\begin{aligned}
    Y_{j+1}
    \le{}&
    C(2^jQ)^{n+q}
    \left(1+\frac{T}{h_R(d_j)}\right)
    [H_R(d_j)]^{-1/\beta'}
    Y_j^{1+1/\beta'}.
\end{aligned}
\]
Moreover, we have
\[
\begin{aligned}
    1+\frac{T}{h_R(d_j)}
    &\le
    2^{(q-1)(j+1)}
    \left(1+\frac{T}{h_R(k)}\right),\\
    [H_R(d_j)]^{-1/\beta'}
    &\le
    2^{q(j+1)/\beta'}[H_R(k)]^{-1/\beta'}.
\end{aligned}
\]
Set
\[
    \nu:=n+q,
    \qquad
    b:=2^{\nu+q-1+q/\beta'}>1.
\]
We obtain
\begin{equation}\label{eq:modular-recurrence}
    Y_{j+1}
    \le C_0b^jQ^\nu[H_R(k)]^{-1/\beta'}
    \left(1+\frac{T}{h_R(k)}\right)
    Y_j^{1+1/\beta'}.
\end{equation}

It remains to choose the level \(k\). Set
\[
    S:=\fint_{B_r(x_0)}H_R(u_+)\dd x,
    \qquad A_0:=(H_R)^{-1}(S),
    \qquad B_0:=(h_R)^{-1}(T).
\]
If \(S=0\) and \(T=0\), then \(u_+=0\) a.e.\ in \(B_r(x_0)\), and the
conclusion is immediate.  Otherwise set
\[
    \kappa_0:=\frac{\nu\beta'}p
\]
choose
\begin{equation}\label{eq:k-choice-explicit}
    k:=KQ^{\kappa_0}(A_0+B_0),
\end{equation}
where \(K\ge1\) will be fixed below.  Since \(h_R(\lambda\tau)\ge\lambda^{p-1}h_R(\tau)\) and \(H_R(\lambda\tau)\ge\lambda^pH_R(\tau)\) for \(\lambda\ge1\),
\begin{equation}\label{eq:k-lower-bounds}
    h_R(k)\ge (KQ^{\kappa_0})^{p-1}T,
    \qquad
    H_R(k)\ge (KQ^{\kappa_0})^pS.
\end{equation}
Indeed, the first inequality follows from \(k\ge KQ^{\kappa_0}B_0\) and \(h_R(B_0)=T\), and the second follows in the same way from \(H_R(A_0)=S\).  In particular,
\[
    \frac{T}{h_R(k)}
    \le (KQ^{\kappa_0})^{-(p-1)}\le1.
\]
Put
\[
    Z_j:=\frac{Y_j}{H_R(k)}.
\]
Dividing \eqref{eq:modular-recurrence} by \(H_R(k)\) and using \(T/h_R(k)\le1\), we obtain
\[
    Z_{j+1}\le C_0b^jQ^\nu Z_j^{1+1/\beta'},
\]
while \(w_0=(u-k)_+\le u_+\) and \eqref{eq:k-lower-bounds} give
\[
    Z_0\le\frac{S}{H_R(k)}
    \le K^{-p}Q^{-p\kappa_0}.
\]
Choose explicitly
\[
    K\ge
    \max\left\{1,
        \left(C_0^{\beta'}b^{(\beta')^2}\right)^{1/p}
    \right\}.
\]
Since \(p\kappa_0=\nu\beta'\), this choice gives
\[
    Z_0
    \le (C_0Q^\nu)^{-\beta'}b^{-(\beta')^2},
\]
which is exactly the threshold in \Cref{lem:fast-convergence} with
\(\gamma=1/\beta'\) and iteration coefficient \(C_0Q^\nu\).  Hence
\(Y_j\to0\).

For completeness, we justify the almost-everywhere conclusion.
Observe that \(B_\rho\subset B_{r_j}\), \(w_j\downarrow(u-2k)_+\), and
\(|B_{r_j}|/|B_\rho|\le2^n\).  Fatou's lemma therefore gives
\[
    \fint_{B_\rho}H_R((u-2k)_+)\dd x
    \le C\liminf_{j\to\infty}Y_j=0.
\]
The strict positivity of \(H_R(\tau)\) for \(\tau>0\) implies
\(u\le2k\) a.e.\ in \(B_\rho(x_0)\). Combining this with
\eqref{eq:k-choice-explicit}, and taking \(\kappa=\kappa_0\) and \(C=2K\),
proves \eqref{eq:concentric-modular}.
\end{proof}

We next establish the modular interpolation estimate that, together with
\Cref{lem3-1}, implies \Cref{thm1}.

\begin{lemma}[Modular interpolation]\label{lem:modular-interpolation}
Let \(0<R\le1\), \(R/2\le r\le R\), and \(\ell>0\).  Let \(v\ge0\) be measurable in \(B_R(x_0)\), with \(v\in L^\ell(B_R(x_0))\cap L^\infty(B_r(x_0))\).
Then, for every \(\delta\in(0,1]\),
\begin{equation}\label{eq:modular-interpolation}
\begin{aligned}
    (H_R)^{-1}
    \left(\fint_{B_r(x_0)}H_R(v)\dd x\right)
    \le{}&
    2\delta\|v\|_{L^\infty(B_r(x_0))}\\
    &+C_\ell(1+\delta^{-\mu})
    \left(\fint_{B_R(x_0)}v^\ell\dd x\right)^{1/\ell},
\end{aligned}
\end{equation}
where
\[
    \mu:=\max\left\{
        \frac{(p-\ell)_+}{\ell},
        \frac{(q-\ell)_+}{\ell}
    \right\},
\]
and \(C_\ell\) depends only on \(n,p,q,\ell\).
\end{lemma}

\begin{proof}
All balls in this proof are centered at \(x_0\).  Put
\[
    P:=\left(\fint_{B_r}v^p\dd x\right)^{1/p},
    \qquad
    Q:=\left(\fint_{B_r}v^q\dd x\right)^{1/q}.
\]
Since
\[
    H_R(P+Q)
    \ge \frac{P^p}{R^{sp}}+a_R^-\frac{Q^q}{R^{tq}}
    =\fint_{B_r}H_R(v)\dd x,
\]
monotonicity gives
\begin{equation}\label{eq:inverse-modular-moments}
    (H_R)^{-1}
    \left(\fint_{B_r}H_R(v)\dd x\right)
    \le P+Q.
\end{equation}
Set
\[
    M:=\|v\|_{L^\infty(B_r)},\qquad
    A:=\left(\fint_{B_R}v^\ell\dd x\right)^{1/\ell}.
\]
Because \(r\ge R/2\), \(|B_R|/|B_r|\le2^n\).  If \(\ell<m\), then
\[
    \fint_{B_r}v^m\dd x
    \le M^{m-\ell}\fint_{B_r}v^\ell\dd x
    \le2^nM^{m-\ell}A^\ell.
\]
If \(\ell\ge m\), H\"older's inequality instead gives
\[
    \left(\fint_{B_r}v^m\dd x\right)^{1/m}
    \le
    \left(\fint_{B_r}v^\ell\dd x\right)^{1/\ell}
    \le2^{n/\ell}A.
\]
Thus, for \(m=p,q\),
\[
    \left(\fint_{B_r}v^m\dd x\right)^{1/m}
    \le
    \begin{cases}
        C M^{1-\ell/m}A^{\ell/m},&\ell<m,\\
        CA,&\ell\ge m.
    \end{cases}
\]
For \(\ell<m\), Young's inequality with conjugate exponents \(m/(m-\ell)\) and \(m/\ell\) yields
\[
    M^{1-\ell/m}A^{\ell/m}
    \le
    \delta M+C\delta^{-(m-\ell)/\ell}A
    \qquad(\ell<m).
\]
Applying this estimate for \(m=p,q\) to \eqref{eq:inverse-modular-moments} proves \eqref{eq:modular-interpolation}.
\end{proof}

We shall use the following standard hole-filling lemma; see \cite{Giu03}.

\begin{lemma}\label{lem:hole-filling}
Let \(R>0\), and let \(\phi\) be nonnegative and bounded on \([R/2,R]\).
Suppose that, for \(R/2\le\rho<r\le R\),
\[
    \phi(\rho)
    \le
    \frac12\phi(r)
    +A\left(\frac{R}{r-\rho}\right)^\gamma+B
\]
with \(A,B\ge0\) and \(\gamma>0\).  Then
\[
    \phi(R/2)\le C_\gamma(A+B).
\]
\end{lemma}

We now prove \Cref{thm1} by combining the preceding three lemmas.

\smallskip

\begin{proof}[\textbf{Proof of \Cref{thm1}}]
Under the present assumptions, \Cref{lem2-1} shows that the minimizer \(u\)
is locally bounded. Hence
\[
    \phi(\tau):=\sup_{B_\tau(x_0)}u_+,
    \qquad \frac R2\le\tau\le R,
\]
is finite and \(u_+\in L^\ell(B_R(x_0))\), as required in
\Cref{lem:modular-interpolation}.  Set
\[
    A:=\left(\fint_{B_R(x_0)}u_+^\ell\dd x\right)^{1/\ell},
    \qquad
    T:=(h_R)^{-1}\bigl(\calT_{R/2,R}(u_+;x_0)\bigr).
\]
Observe that both \(A\) and \(T\) are independent of the intermediate radii
introduced below.

Fix \(R/2\le\rho<r\le R\), and put
\[
    Q:=\frac{R}{r-\rho}.
\]
Since \(r-\rho\le R/2\), we have \(Q\ge2\).  Applying \Cref{lem3-1} to the concentric pair
\(B_\rho(x_0)\subset B_r(x_0)\) yields
\begin{equation}\label{eq:upper-proof-modular}
    \phi(\rho)
    \le C Q^\kappa
    \left[
        \bigl(H_R\bigr)^{-1}
        \left(\fint_{B_r(x_0)}H_R(u_+)\dd x\right)
        +T
    \right].
\end{equation}
The constants \(C\) and \(\kappa\) depend only on the data in
\(\mathfrak D_u(B_R(x_0))\).

Let \(0<\delta\le1\). Applying \Cref{lem:modular-interpolation} with
\(v=u_+\), we obtain
\[
    \bigl(H_R\bigr)^{-1}
    \left(\fint_{B_r(x_0)}H_R(u_+)\dd x\right)
    \le
    2\delta\,\phi(r)
    +C_\ell(1+\delta^{-\mu})A,
\]
where \(\mu\ge0\) is defined in \Cref{lem:modular-interpolation} and depends only on \(p,q,\ell\).  Substituting into~\eqref{eq:upper-proof-modular}, we obtain
\begin{equation}\label{eq:upper-proof-preiteration}
    \phi(\rho)
    \le
    C Q^\kappa
    \left[2\delta\phi(r)+C_\ell(1+\delta^{-\mu})A+T\right].
\end{equation}

Now we choose
\[
    \delta:=\frac{1}{4C Q^\kappa}.
\]
Since \(Q\ge2\) and \(C\ge1\), this choice belongs to \((0,1]\), and the coefficient of \(\phi(r)\) in~\eqref{eq:upper-proof-preiteration} becomes \(1/2\).  Furthermore,
\[
    C Q^\kappa(1+\delta^{-\mu})
    \le C Q^\kappa\bigl(1+(4C Q^\kappa)^\mu\bigr)
    \le C Q^{\kappa(1+\mu)}.
\]
The tail term satisfies \(CQ^\kappa T\le CQ^\gamma T\) for every
\(\gamma\ge\kappa\).  Hence, with
\[
    \gamma:=\kappa(1+\mu),
\]
we arrive at
\begin{equation}\label{eq:upper-proof-hole-filling}
    \phi(\rho)
    \le \frac12\phi(r)
       +C_\ell\left(\frac{R}{r-\rho}\right)^\gamma(A+T),
\end{equation}
for every \(R/2\le\rho<r\le R\).  All constants depend only on \(\ell\)
and \(\mathfrak D_u(B_R(x_0))\).

Finally, apply \Cref{lem:hole-filling}
to~\eqref{eq:upper-proof-hole-filling} with \(B=0\). Since \(\phi\) is bounded
on \([R/2,R]\), we obtain
\[
    \phi(R/2)\le C_\ell(A+T).
\]
Recalling the definitions of \(\phi\), \(A\), and \(T\), this is precisely~\eqref{eq:local-bound}.
\end{proof}

We conclude the section with the proofs of \Cref{thm:harnack-tail} and
\Cref{cor:harnack-single-tail}, obtained by combining the weak Harnack
estimate in \Cref{pro1} with the upper estimate in \Cref{thm1}.

\smallskip

\begin{proof}[\textbf{Proof of \Cref{thm:harnack-tail}}]
Let \(\varepsilon\in(0,1)\) be the exponent supplied by
\Cref{pro1}.  We apply \Cref{thm1} with the
reference radius replaced by \(R/2\) and with \(\ell=\varepsilon\). Since
\(u\ge0\) a.e.\ in \(B_{2R}(x_0)\), we have
\(u_+=u\) a.e.\ in \(B_{R/2}(x_0)\). Therefore
\begin{equation}\label{eq:harnack-upper-step}
\begin{aligned}
    \sup_{B_{R/4}(x_0)}u
    \le C\left(\fint_{B_{R/2}(x_0)}u^\varepsilon\dd x\right)^{1/\varepsilon}+C(h_{R/2})^{-1}
    \bigl(\calT_{R/4,R/2}(u_+;x_0)\bigr).
\end{aligned}
\end{equation}
The constant in~\eqref{eq:harnack-upper-step} depends on \(\varepsilon\) and on the local data associated with \(B_{R/2}(x_0)\).  These local data are controlled by \(\mathfrak D_u(B_{2R}(x_0))\).  Moreover, \(\varepsilon\) itself depends only on \(\mathfrak D_u(B_{2R}(x_0))\) by
\Cref{pro1}.  Hence the constant in~\eqref{eq:harnack-upper-step} has the dependence asserted in \Cref{thm:harnack-tail}.

On the other hand, \Cref{pro1} gives
\begin{equation}\label{eq:harnack-lower-step}
    \left(\fint_{B_{R/2}(x_0)}u^\varepsilon\dd x\right)^{1/\varepsilon}
    \le
    C\left[
        \inf_{B_{R/2}(x_0)}u
        +(h_{2R})^{-1}\bigl(\calT_{2R}(u_-;x_0)\bigr)
    \right].
\end{equation}
Substituting~\eqref{eq:harnack-lower-step}
into~\eqref{eq:harnack-upper-step} yields
\begin{align*}
    \sup_{B_{R/4}(x_0)}u
    \le C\bigg[
        \inf_{B_{R/2}(x_0)}u
        +(h_{R/2})^{-1}
          \bigl(\calT_{R/4,R/2}(u_+;x_0)\bigr)+(h_{2R})^{-1}
          \bigl(\calT_{2R}(u_-;x_0)\bigr)
    \bigg].
\end{align*}
Finally, because \(B_{R/4}(x_0)\subset B_{R/2}(x_0)\),
\[
    \inf_{B_{R/2}(x_0)}u
    \le \inf_{B_{R/4}(x_0)}u.
\]
Replacing the smaller quantity on the right-hand side by the larger one yields~\eqref{eq:harnack-tail}.
\end{proof}


\begin{proof}[\textbf{Proof of \Cref{cor:harnack-single-tail}}]
For brevity, set
\[
    A:=\calT_{R/4,R/2}(u_+;x_0),
    \qquad
    B:=\calT_{2R}(u_-;x_0),
    \qquad
    T:=\calT_{R/4,2R}(|u|;x_0).
\]
The inclusion \(B_{R/2}(x_0)\times B_{R/2}(x_0)\subset B_{2R}(x_0)\times B_{2R}(x_0)\) implies \(a_{R/2}^-\ge a_{2R}^-\).  Thus, for every $\tau\ge0$,
\[
\begin{aligned}
    h_{R/2}(\tau)
    &=\frac{\tau^{p-1}}{(R/2)^{sp}}
      +a_{R/2}^-\frac{\tau^{q-1}}{(R/2)^{tq}}\\
    &\ge
      \frac{\tau^{p-1}}{(2R)^{sp}}
      +a_{2R}^-\frac{\tau^{q-1}}{(2R)^{tq}}
      =h_{2R}(\tau).
\end{aligned}
\]
Since both functions are strictly increasing bijections of
\([0,\infty)\), their inverses satisfy
\[
    (h_{R/2})^{-1}(S)\le (h_{2R})^{-1}(S)
    \qquad\text{for every }S\ge0.
\]
The monotonicity of \((h_{2R})^{-1}\) therefore yields
\begin{equation}\label{eq:inverse-tail-comparison}
\begin{aligned}
    &\quad(h_{R/2})^{-1}(A)+(h_{2R})^{-1}(B)\\
    &\le
    (h_{2R})^{-1}(A)+(h_{2R})^{-1}(B)
    \le 2(h_{2R})^{-1}(A+B).
\end{aligned}
\end{equation}

The sign assumption in \Cref{thm:harnack-tail} gives \(u_-=0\) a.e.\ in
\(B_{2R}(x_0)\). Hence enlarging the exterior integration region from
\(\R^n\setminus B_{2R}(x_0)\) to \(\R^n\setminus B_{R/4}(x_0)\) adds only a region on which \(u_-=0\), so
\[
    \calT_{2R}(u_-;x_0)
    =\calT_{R/4,2R}(u_-;x_0),
\]
where the interior reference ball is \(B_{2R}(x_0)\) on both sides. Moreover, enlarging the ball over which the weighted \(q\)-tail supremum
is taken gives
\[
    \calT_{R/4,R/2}(u_+;x_0)
    \le \calT_{R/4,2R}(u_+;x_0).
\]

To verify the resulting signed-tail comparison directly, define
\[
\begin{aligned}
    P(v)
    &:={}
      \int_{\R^n\setminus B_{R/4}(x_0)}
      \frac{|v(y)|^{p-1}}{|x_0-y|^{n+sp}}\dd y,\\
    Q(v;x)
    &:={}
      \int_{\R^n\setminus B_{R/4}(x_0)}
      a(x,y)\frac{|v(y)|^{q-1}}{|x_0-y|^{n+tq}}\dd y,
      \qquad x\in B_{2R}(x_0).
\end{aligned}
\]
The preceding two observations imply
\[
    A+B
    \le P(u_+)+P(u_-)+
       \sup_{x\in B_{2R}(x_0)}Q(u_+;x)
       +\sup_{x\in B_{2R}(x_0)}Q(u_-;x).
\]
Since \(|u|^{m-1}=u_+^{m-1}+u_-^{m-1}\) for \(m=p,q\), we have
\(P(u_+)+P(u_-)=P(|u|)\) and
\(Q(u_+;x)+Q(u_-;x)=Q(|u|;x)\). It follows that
\begin{equation}\label{eq:signed-tail-comparison}
\begin{aligned}
    A+B
    &\le P(|u|)+2\sup_{x\in B_{2R}(x_0)}Q(|u|;x)\\
    &\le2\calT_{R/4,2R}(|u|;x_0)=2T.
\end{aligned}
\end{equation}

Since \(q\ge p\), for every \(\lambda\ge1\) and \(\tau\ge0\),
\[
\begin{aligned}
    h_{2R}(\lambda\tau)=\lambda^{p-1}\frac{\tau^{p-1}}{(2R)^{sp}}
      +a_{2R}^-\lambda^{q-1}\frac{\tau^{q-1}}{(2R)^{tq}}
    \ge \lambda^{p-1}h_{2R}(\tau).
\end{aligned}
\]
Taking \(\tau=(h_{2R})^{-1}(S)\) and
\(\lambda=2^{1/(p-1)}\), and then applying the increasing inverse, gives
\[
    (h_{2R})^{-1}(2S)
    \le 2^{1/(p-1)}(h_{2R})^{-1}(S)
    \qquad\text{for every }S\ge0.
\]
Combining this with \eqref{eq:inverse-tail-comparison} and \eqref{eq:signed-tail-comparison}, we obtain the explicit estimate
\[
\begin{aligned}
    &\quad (h_{R/2})^{-1}(A)+(h_{2R})^{-1}(B)\\
    &\le2(h_{2R})^{-1}(A+B)
     \le2(h_{2R})^{-1}(2T)
     \le2^{1+1/(p-1)}(h_{2R})^{-1}(T).
\end{aligned}
\]
Substitution of this bound into \eqref{eq:harnack-tail} proves \eqref{eq:harnack-single-tail}. 
\end{proof}

\subsection*{Acknowledgments}
This work was supported by the National Natural Science Foundation of China
(Nos.~12471128 and 12301245) and the Natural Science Foundation of
Heilongjiang Province (No.~YQ2025A002).

\subsection*{Conflict of interest}
The authors declare that they have no conflict of interest.

\subsection*{Data availability}
Data sharing is not applicable to this article because no datasets were
generated or analyzed during the current study.

\end{document}